\documentclass{amsart}

\usepackage{amsmath,amsfonts,amssymb,amsthm,pb-diagram,picinpar,graphicx,color}
\usepackage{colortbl}
\usepackage{enumerate}
\PassOptionsToPackage{hyphens}{url}
\usepackage{hyperref}
\usepackage{bm}
\usepackage{tikz}
\usepackage{pgf}
\usepackage[T1]{fontenc}
\usepackage{enumitem}

\usetikzlibrary{shapes,matrix,calc,arrows,decorations.markings,knots,patterns,intersections,cd}

\newcommand{\CC}{\mathbb{C}}
\newcommand{\RR}{\mathbb{R}}
\newcommand{\ZZ}{\mathbb{Z}}
\newcommand{\PP}{\mathbb{P}}
\newcommand{\mcC}{\mathcal{C}}
\newcommand{\mcB}{\mathcal{B}}
\newcommand{\mcO}{\mathcal{O}}

\newcommand{\fd}{\mathfrak{d}}

\newcommand{\fm}{\mathfrak{m}}
\newcommand{\ft}{\mathfrak {t}}
\newcommand{\fo}{\mathfrak{o}}
\newcommand{\fS}{\mathfrak{S}}

\newcommand{\WCC}{\underline{\mcC}}
\newcommand{\WCB}{\underline{\mcB}}

\newcommand{\pic}{\mathop{\mathrm{Pic}}\nolimits}
\newcommand{\Div}{\mathop{\mathrm{Div}}\nolimits}
\newcommand{\irr}{\mathop{\mathrm{Irr}}\nolimits}

\newcommand{\str}{\mathop{\mathrm{Str}}\nolimits}

\newcommand{\sing}{\mathop{\mathrm{Sing}}\nolimits}

\newcommand{\ord}{\mathop{\mathrm{ord}}\nolimits}
\newcommand{\merid}{\mathfrak{m}}

\newcommand{\comb}{\mathop{\mathrm{Comb}}\nolimits}
\newcommand{\blo}{\mathrm{bl}}
\DeclareMathOperator{\coker}{coker}

\newtheorem{thm}{Theorem}[section]
\newtheorem{lem}[thm]{Lemma}     
\newtheorem{cor}[thm]{Corollary}
\newtheorem{prop}[thm]{Proposition}

\theoremstyle{definition}
\newtheorem{defin}[thm]{Definition}

\newtheorem{ex}[thm]{Example}

\theoremstyle{remark}
\newtheorem{rem}[thm]{Remark}

\title
{
Torsion divisors of plane curves  and Zariski pairs 
}
\date{\today}

\author[E. Artal]{Enrique Artal Bartolo}
\address{
Departamento de Matema{\'a}ticas, IUMA, Universidad de Zaragoza, C. Pedro Cerbuna 12, 50009 Zaragoza, SPAIN}
\email{artal@unizar.es}

\author[S. Bannai]{Shinzo Bannai}
\address{Department of Natural Sciences, National Institute of Technology, Ibaraki College, 866 Nakane, Hitachinaka, Ibaraki 312-8508, Japan}
\email{sbannai@ge.ibaraki-ct.ac.jp}

\author[T. Shirane]{Taketo Shirane}
\address{Department of Mathematical Sciences, Faculty of Science and Technology, Toku\-shima University, Tokushima, 770-8502, Japan}
\email{shirane@tokushima-u.ac.jp}

\author[H. Tokunaga]{Hiro-o TOKUNAGA}
\address{Department of Mathematical Sciences, Graduate School of Science,
Tokyo Metropolitan University,
1-1 Minami-Ohsawa, Hachiohji 192-0397 JAPAN}
\email{tokunaga@tmu.ac.jp}

\thanks{First named author is partially supported by PID2020-114750GB-C31 / AEI/ 10.13039/501100011033 and Departamento de Ciencia, Universidad y Sociedad del Conocimiento of the Gobierno de Arag{\'o}n (E22\_20R:
``{\'A}lgebra y Geometr{\'i}a''). Second named author is partially supported by Grant-in-Aid for Scientific Research C (18K03263). Third named author is partially supported by Grant-in-Aid for Scientific Research C (21K03182).  Fourth named author is partially supported by Grant-in-Aid for Scientific Research C (17K05205).}

\begin{document}

\maketitle

\begin{abstract}
In this paper we study the embedded topology of reducible plane curves having a smooth irreducible component.   In previous studies, the relation between the topology and certain torsion classes in the Picard group of degree zero of the smooth component was implicitly considered. We formulate this relation clearly and  give a criterion for distinguishing the embedded topology in terms of torsion classes. Furthermore, we give a method of systematically constructing examples of curves where our criterion is applicable, and give new examples of Zariski tuples.  
\end{abstract}

Zariski pairs. 
\medskip


\section*{Introduction}
\label{intro}

The embedded topology of plane curves has been studied by many mathematicians since the initial work of Zariski \cite{zariski}. 
One of the main questions about the embedded topology of plane curves
is its relationship with the combinatorics of the curve (roughly speaking, decomposition in irreducible components with their degrees and topological type of singular points, see Definition~\ref{def:combinatorics}). This is because, the combinatorics
determine the topology of the regular neighborhood of the curve, but as is demonstrated in Zariski's example, it is not enough to determine the embedded topology.  Zariski's example also demonstrates \cite{zariski} that the \emph{geometric} position of singular points have a surprisingly  strong influence on the embedded topology (see~\cite{survey} for a detailed survey and references therein). 
A finite set of plane curves with the same combinatorics but different embedded topology with each other is called a \textit{Zariski tuple} (see Definition~\ref{def:zariski pair}). 

The first main invariant of the embedded topology is the fundamental group of the complement (cf. \cite{zariski, oka1996, eyral-oka}), which is very powerful but at the same time presents some issues. Its computation is usually difficult (where one aims to obtain a presentation of the fundamental group in terms of generators and relations)  and some times getting a presentation does not always provide relevant information, due to the difficulty in solving the so called "word problem". This is why other more effective (and algebraic) invariants have been used, such as Alexander polynomials \cite{libgober1982}  (involving cyclic covers) or the existence/non-existence of certain dihedral covers  \cite{tokunaga1994}. Later on, other finer invariants have been used such as characteristic varieties  \cite{libgober2001}
or $K3$-lattice invariants (see \cite{degtyarev, shimada}). 
Recently,  new kinds of invariants called \textit{linking invariants} (\cite{afg, benoit-meilhan}) and \textit{splitting invariants} (\cite{tokunaga1, bannai, shirane2016, shirane2019}) have been introduced for reducible curves, whose definitions do not involve the fundamental group (here by splitting invariants we mean invariants derived from splitting phenomena of plane curves under Galois covers). 
Linking invariants are derived from geometric topology, and splitting invariants are defined in terms of algebraic geometry. 
Both invariants essentially represent how a component of the reducible curve "tangles up" with the other components (cf. \cite{benoit-meilhan}, \cite{ben-shi}). 

The first interesting degree for the existence
of Zariski pairs or tuples is~$6$; the first cases of Zariski
pairs were found by several authors using Alexander polynomials
but for the final classification more delicate invariants
were needed, such as $K3$-lattice invariants, see
\cite{shimada,degtyarev}. The existence
of \emph{arithmetic} Zariski pairs (curves  defined 
over a number field $K$ whose topology depends
on the embedding $K\hookrightarrow\CC$) shows that
one needs to go beyond algebraic properties.

A particularly interesting case is given for curves having only smooth
irreducible components, e.g. line arrangements.
For line arrangements, the combinatorics has a simpler interpretation;
a famous example of Rybnikov~\cite{ryb} is a Zariski pair distinguished by the fundamental group. 
Some other Zariski pairs of line arrangements were obtained by other authors by using braid monodromy
or deep properties of the fundamental group. The Zariski pairs found 
in \cite{GB:16,GBVS} used a linking invariant. Conic-line arrangements and conic arrangements have also been studied \cite{NT, tokunaga14, tokunaga-bannai16}. 

Another family of interesting cases comes from the arrangements
of smooth curves allowing positive genus. A classical example
was studied in \cite{artal94}:
a smooth cubic $D$ and $3$ inflectional tangents of $D$, later
generalized to conic-cubic arrangements in~\cite{ac:98}. In these cases, the fundamental group is the key invariant to distinguish 
Zariski pair candidates. The case of more inflectional tangents
is studied in \cite{bgst} by using the linking and splitting invariants. These Zariski pairs (or tuples) are characterized by the position of the tangent points of the smooth cubic $D$ and the inflectional tangents. 
The linking and splitting invariants were also applied to give Zariski tuples of one smooth quartic and its bitangents in \cite{benoit-meilhan, bannai-ohno}. 
Other cases with two smooth non-linear components has been studied by I.~Shimada \cite{shimada2003},  the third named author~\cite{shirane2016} and the fourth named author~\cite{tokunaga1}.  The case of
one smooth curve~$D$ of arbitrary degree and an arrangement of lines
(usually with some tangency relation with $D$) has been studied
by the first named author and collaborators in \cite{acm} and
by the third named author in~\cite{shirane2019}. 
The linking and splitting invariants  have been used in these cases and an important
feature  is that for all of these curves  (besides small degree examples) the fundamental groups of their complements are abelian,
certifying the independence
of the linking invariants and the fundamental group. More recently, the second and fourth named authors considered the case of multiple pairs of  simple tangent lines to a cubic in \cite{bannai-tokunaga2019}. As in the case of inflectional points, the study is also made in terms of the geometric group structure of the cubic. 

In this paper, we consider the embedded topology of plane curves consisting of a smooth curve $D$ and some additional  possibly reducible curve $\WCB$ on the complex projective plane $\PP^2$. The key point
is to relate the topology of the arrangement with the torsion
classes of the subgroup $\pic^0(D)\subset \pic(D)$ of divisor classes of degree  zero of the smooth curve $D$. Our idea is
to formulate a criterion to distinguish the embedded topology of curves explicitly in terms of these torsion classes. 
We fix a large family
of cyclic covers and 
we construct a function associating each cyclic cover
to a \emph{splitting number} (one of the splitting invariants) for providing a criterion to distinguish
candidate Zariski pairs.

This criterion is useful in two directions.
On one side it allows us to re-explain previous results
in the literature in a uniform setting,
allowing us to distinguish the embedded topology of curves in a more finer way (cf. Section~\ref{sec:remarks}). 
On the other side, we are able to construct new Zariski tuples by using this criterion (Section~\ref{sec:example}). 

We note that the assumption that $D$ is smooth is somewhat technical (by this assumption $\pic^0(D)$ is easier to deal with) but we 
plan to generalize to the case where $D$ is allowed to be singular in the future. 
We have already a work in progress for the special case where $D$ is a nodal cubic and we expect to deal also with the case
of $D$ singular and non-rational.

We organize this paper as follows. 
In Section~\ref{sec:main}, we state the main results of this paper. 
In Section~\ref{sec:setting}, we recall the definition and basic properties of combinatorics of plane curves and certain subgroups of $\Div(D)$. 
In Section~\ref{sec:splitting number}, we give a formula to calculate splitting numbers in terms of the order of certain elements of $\pic^0(D)$, and formulate a criterion (Corollary~\ref{cor:n-torsion}) to distinguish the embedded topology of plane curves by using the formula. 

In Section~\ref{sec:remarks}, we apply   the criterion 
to some previous results. 
In Section~\ref{sec:example}, we give a method of systematically constructing examples of curves which are generalizations of curves studied by Shimada \cite{shimada2003} where the criterion is applicable, and give a new Zariski $4$-tuple (Theorem~\ref{thm:example}). 

\section{Main results}\label{sec:main}
The plane curves which are considered in this paper are of the following form: 
\[ \WCC:=D+\mcC_1+\dots+\mcC_k, \]
where $D$ is a smooth curve of degree $d_0$, and $\mcC_j$ is a possibly reducible curve of degree $d_j$ for each $j=1,\dots,k$. 
Since $\mcC_1,\dots,\mcC_k$ are possibly reducible, there may be different decompositions of $\WCC$. 
We describe the above decomposition of $\WCC$ by $[\WCC]:=(D;\mcC_1,\dots,\mcC_k)$. 
First, we describe how we obtain torsion classes of $\pic^0(D)$ from a decomposition $[\WCC]$ of $\WCC$ as above. 
We denote the subgroup of $\pic^0(D)$ consisting of all $n$-torsion classes by $\pic^0(D)[n]$; for $m\in\mathbb{N}$ we denote by $\Div(D)_m$ the subgroup of divisors
whose degree is divisible by~$m$: 
\begin{align*}
	\pic^0(D)[n]&:=\{ \mathfrak{c}\in\pic^0(D) \mid n\mathfrak{c}= 0  \mbox{ in $\pic^0(D)$} \},
	\\
	\Div(D)_m &:=\{ \fd\in\Div(D) \mid \deg\fd\equiv0\pmod{m}
	\}.
\end{align*}
Put $\fo_D:=L|_D$ for a line $L\subset\PP^2$. 
Note that $\fo_D$ is a divisor of degree $d_0$ on $D$. 
Put 
\[ n_{[\WCC]}:=\gcd\Big(
\{(\mcC_{j}\cdot D)_P\mid P\in \mcC_{j}\cap D, \ j=1,\dots,k\}\cup\{d_1,\dots,d_k\}\Big)
. \]
We write $n=n_{[\WCC]}$ by omitting the subscript $[\WCC]$ for short if there is no confusion. 
There uniquely exists an effective divisor $\fd_j:=\fd_j[\WCC]$ on $D$ such that $n\fd_j={\mcC_j}|_D$ for each $j=1,\dots,k$. 
Let $\ft_{j}[\WCC]:=\fd_{j}-\frac{d_{j}}{n}\fo_{D}$. 
The divisor $\ft_j:=\ft_j[\WCC]$ satisfies 
\[ n\ft_j=n\left(\fd_j-\frac{d_j}{n}\fo_D\right)={\mcC_j}|_D-d_j\fo_D=(\mcC_j-d_jL)|_D\sim 0. \]
By abuse of notation, we denote the divisor class by the same symbol of its representative, for example, $\ft_{j}$  will also describe the element of $\pic^0(D)[n]$ containing the divisor $\ft_j$. 

Let $G[\WCC]$ be the subgroup of $\pic^0(D)[n]$ generated by $\ft_{1},\dots,\ft_{k}\in\pic^0(D)[n]$. 
We define 
$\tau_{[\WCC]}:\ZZ^{\oplus k}\twoheadrightarrow G[\WCC]$ as the homomorphism given by  
\[
\tau_{[\WCC]}(a_1,\dots,a_k)=a_1\ft_{1}+\dots+a_k\ft_{k}. 
\]
Let $\bm{d}:=(d_1,\dots,d_k)$; a permutation $\rho$ of $k$ letters is said to be \textit{admissible} with respect to $\bm{d}$ if $d_j=d_{\rho(j)}$ for each $j$. 
We consider permutations with  additional conditions later (see Subsection~\ref{subsect:combinatorics}), however the definition of admissibility above is simpler, and we use this admissibility above to keep the statement of  Proposition~\ref{prop:n-torsion} simpler. 
Note that the symmetric group $\fS_k$ acts on $\ZZ^{\oplus k}$ by 
\[ \rho(a_1,\dots,a_k)=\left(a_{\rho^{-1}(1)},\dots,a_{\rho^{-1}(j)},\dots,a_{\rho^{-1}(k)}\right) \]
for $\rho\in\fS_k$ and $(a_1,\dots,a_k)\in\ZZ^{\oplus k}$, and $\rho$ can be regarded as a map $\rho:\ZZ^{\oplus k}\to~\ZZ^{\oplus k}$.  

\begin{prop}\label{prop:n-torsion}
Let $\WCC_i$ $(i=1,2)$ be two plane curves 
\[ \WCC_i:=D_i+\mcC_{i1}+\dots+\mcC_{ik} \]
where $D_1$ and $D_2$ are smooth curves of degree $d_0$, and $\mcC_{ij}$ are plane curves of $\deg\mcC_{1j}=\deg\mcC_{2j}=d_j$. 
Put $[\WCC_i]:=(D_i;\mcC_{i1},\dots,\mcC_{ik})$. 
Assume $n_{[\WCC_1]}=n_{[\WCC_2]}$. 
Let $\rho$ be an admissible permutation with respect to $\bm{d}$. 
If $\ker \tau_{[\WCC_1]}\ne\ker\tau_{[\WCC_2]}\circ\rho$, then there exists no homeomorphism $h:\PP^2\to\PP^2$ such that $h(D_1)=D_2$ and $h(\mcC_{1j})=\mcC_{2\rho(j)}$. 
In particular, if $G[\WCC_1]$ and $G[\WCC_2]$ are not isomorphic, then there exist no homeomorphism $h:\PP^2\to\PP^2$ and no permutation $\rho$ of $k$ letters such that $h(D_1)=D_2$ and $h(\mcC_{1j})=\mcC_{2\rho(j)}$. 
\end{prop}

In the case when $D$ has a special point such as a maximal flex,
Proposition~\ref{prop:n-torsion} can be formulated in a different way which will be discussed in~\cite{cubic_torsion}.
Proposition  \ref{prop:n-torsion} tells us about the existence/non-existence of certain homeomorphisms. However, in order to apply it and to give a criterion to distinguish Zariski tuples, we need to consider the combinatorics of the curves in detail which can be quite subtle. 
See Subsection~\ref{subsect:combinatorics} for the definitions and notation used in Theorem~\ref{thm:n-torsion} and Corollary~\ref{cor:n-torsion}.

\begin{thm}\label{thm:n-torsion}
Let $\WCC_i:=D_i+\mcC_{i1}+\dots+\mcC_{ik}$ ($i = 1, 2$) be two plane curves satisfying the hypotheses of Proposition{\rm~\ref{prop:n-torsion}} and having the same combinatorics. 
We assume that any equivalence map $\varphi:\comb(\WCC_1)\to\comb(\WCC_2)$ is admissible to $([\WCC_1],[\WCC_2])$. 
Then if $\ker\tau_{[\WCC_1]}\ne\ker\tau_{[\WCC_2]}\circ\rho$ for any admissible permutation $\rho$ to $([\WCC_1],[\WCC_2])$, then $(\WCC_1,\WCC_2)$ is a Zariski pair. 
\end{thm}

In particular, we have the following:

\begin{cor} \label{cor:n-torsion}
Assume the same hypotheses as Theorem~\ref{thm:n-torsion}, and furthermore put $\ft_{i,j}:=\ft_{j}[\WCC_i]\in\pic^0(D_i)[n]$ for $i=1,2$ and $j=1,\dots,k$. Then the following statements hold:
\begin{enumerate}[label=\rm(\roman{enumi})]

 \item \label{cor:n-torsion-iii} 
 If $G[\WCC_1]$ and $G[\WCC_2]$ are not isomorphic, 
 then $(\WCC_1, \WCC_2)$ is a Zariski pair.

 \item\label{cor:n-torsion-i} If $k=1$ and $\ord(\ft_{1,1})\not= \ord(\ft_{2,1})$,
 then $(\WCC_{1}, \WCC_{2})$ is a Zariski pair.
 
 \item \label{cor:n-torsion-ii} If $(\ord(\ft_{1,1}),\dots,\ord(\ft_{1,k}))\not= (\ord(\ft_{2,\rho(1)}),\dots,\ord(\ft_{2,\rho(k)}))$ for any permutation $\rho$ admissible to $([\WCC_1],[\WCC_2])$,
 then $(\WCC_1, \WCC_2)$ is a Zariski pair.

 \end{enumerate}

 \end{cor}
 
 \begin{rem}
 Note that Corollary~\ref{cor:n-torsion} can be applied to Zariski pairs previously given
 as follows:
 \begin{itemize} 
 
 \item A Zariski pair for $3$-Artal arrangements \cite{artal94}
 follows from Corollary~\ref{cor:n-torsion}~\ref{cor:n-torsion-i}.
 
 \item  A Zarsiki pair for a smooth cubic and its tangent lines considered in \cite{bannai-tokunaga2019}
 follows from Corollary~\ref{cor:n-torsion}~\ref{cor:n-torsion-iii}.

\end{itemize}
We will explain these   examples in   more detail in Section~\ref{sec:remarks}.
\end{rem}

Theorem~\ref{thm:n-torsion} and Corollary~\ref{cor:n-torsion} enable us to distinguish the embedded topology simply. 
In Section~\ref{sec:example}, we consider plane curves whose embedded topology can be distinguished by Corollary~\ref{cor:n-torsion}. 
We also give a method to systematically construct more complicated examples based on the data of simple cases. 
In particular, we construct the following Zariski $4$-tuple. 

\begin{thm}\label{thm:example}
	There exists a Zariski $4$-tuple of plane curves of the form $\WCC=D+C$, where $D$ and $C$ are smooth curves of degree $4$ and $6$, respectively, and $(D\cdot C)_P=6$ for any $P\in D\cap C$. 
\end{thm}

\section{Settings}\label{sec:setting}

\subsection{Combinatorics and Zariski tuples}\label{subsect:combinatorics}

The definition of the combinatorics of a curve is given in \cite{survey}, but we give it here for sake of completeness and the convenience of the reader. We also formulate the notion of {\it equivalence} of  combinatorial types which was intrinsically considered in  \cite{survey}, but not explicitly given there.

\begin{defin}\label{def:combinatorics}
Let $\mcC$ be a plane curve in $\PP^2$. Let $\blo_\mcC:\hat\PP^2\to\PP^2$ be the minimal embedded resolution of its set 
$\sing_\mcC$ of singular points
(i.e., the minimal sequence of blowing-ups which resolve $\sing_\mcC$ and such that $\blo_\mcC^{-1}(\mcC)$ 
is a simple normal crossing divisor).
Let $\irr_\mcC$ be the set of irreducible components of $\mcC$, and let $(\Gamma_\mcC,\str_\mcC,e_\mcC)$ be the $3$-tuple given by
\begin{enumerate}
\item $\Gamma_\mcC$ is the dual graph of the simple normal crossing curve $\blo_\mcC^{-1}(\mcC)$ with set of vertices $\mathcal{V}_\mcC$.
\item $\str_\mcC$ is the set of vertices corresponding to the strict transforms of the irreducible components of $\mcC$ (in natural bijection
with $\irr_\mcC$).
\item The Euler map $e_\mcC:\mathcal{V}_\mcC\to\mathbb{Z}$ is the map of self-intersections.
\end{enumerate}
Then $\comb(\mcC):=(\Gamma_\mcC,\str_\mcC,e_\mcC)$ is the \emph{combinatorial type} (or \emph{combinatorics}) of the curve $\mcC$.
\end{defin}

\begin{rem}
The degrees and genera of the irreducible components of $\mcC$ can be derived from their combinatorics. 
Conversely, the combinatorics of a plane curve $\mcC$ is determined by combinatorial data of $\mcC$ (degrees of curves in $\irr_\mcC$, topological types of singularities of $\mcC$, and so on). 
Since $\str_\mcC$ and $\irr_\mcC$ have a canonical bijection we will identify them
as far as no confusion arises.
\end{rem}

\begin{defin}
Let $\mcC_1,\mcC_2$ be plane curves in $\PP^2$. 
\begin{enumerate}
\item An \emph{equivalence map} $\varphi:\comb(\mcC_1)\to\comb(\mcC_2)$ 
is an isomorphism $\tilde\varphi:\Gamma_{\mcC_1}\to\Gamma_{\mcC_2}$ of graphs such that $\tilde\varphi(\str_{\mcC_1})=\str_{\mcC_2}$ and $e_{\mcC_1}=e_{\mcC_2}\circ\tilde\varphi$. 

\item The curves $\mcC_1,\mcC_2$ have the \emph{same combinatorial type} if there exists an equivalence map 
$\varphi:\comb(\mcC_1)\to\comb(\mcC_2)$. 

\item A homeomorphism $h:(\PP^2,\mcC_1)\to(\PP^2,\mcC_2)$
induces an equivalence map $\varphi_h:\comb(\mcC_1)\to\comb(\mcC_2)$ called the \emph{equivalence map induced by $h$}.
\end{enumerate}
\end{defin}

\begin{rem}
An equivalence map $\varphi:\comb(\mcC_1)\to\comb(\mcC_2)$ induces canonical bijections $\varphi_{\irr}:\irr_{\mcC_1}\to\irr_{\mcC_2}$ and $\varphi_{\sing}:\sing_{\mcC_1}\to\sing_{\mcC_2}$, where $\sing_{\mcC_i}$ is the set of singular points of $\mcC_i$. 
Note that $\varphi_{\irr}$ and $\varphi_{\sing}$ preserve degrees of irreducible components and topological types of singularities, respectively. 
\end{rem}

Given curves $\WCC_i$ and decompositions $[\WCC_i]$ $(i=1,2)$ as in Section \ref{sec:main} and a homeomorphism $h:\PP^2\rightarrow \PP^2$ satisfying $h(\WCC_1)=\WCC_2$,  
$h$ induces an equivalence map of combinatorial types $\varphi_h: \comb(\WCC_1) \rightarrow \comb(\WCC_2)$. However, this equivalence need not preserve the structure of the decompositions $[\WCC_i]$ in general, i.e. $\{h(\mcC_{11}),\ldots, h(\mcC_{1k_1})\}\not=\{\mcC_{21}, \ldots, \mcC_{2k_2}\}$ in general. Since our arguments later depend on the decompositions $[\WCC_i]$ we need to take this fact into account, which leads to the notion of admissibility of equivalence maps and permutations used in Theorem~\ref{thm:n-torsion} and Corollary~\ref{cor:n-torsion}. They are defined as follows:

\begin{defin}
Let $\WCC_i:=D_i+\mcC_{i,1}+\dots+\mcC_{i,k}$ ($i=1,2$) be two plane curves as in Proposition~\ref{prop:n-torsion}, and put $[\WCC_i]:=(D_i;\mcC_{i,1},\dots,\mcC_{i,k})$. 
Assume that $\WCC_1$ and $\WCC_2$ have the same combinatorial type. 
\begin{enumerate}

\item An equivalence map $\varphi\!:\comb(\WCC_1)\to\comb(\WCC_2)$ is \textit{admissible to $([\WCC_1],[\WCC_2])$} if $\varphi_{\irr}(D_1)=D_2$ and there exists a permutation $\rho_\varphi$ of $k$ letters such that $\varphi_{\irr}(\irr_{\mcC_{1,j}})=\irr_{\mcC_{2,\rho_\varphi(j)}}$ for any $j=1,\dots,k$. 

\item A permutation $\rho$ of $k$ letters is \textit{admissible to $([\WCC_1],[\WCC_2])$} if there exists an equivalence map $\varphi:\comb(\WCC_1)\to\comb(\WCC_2)$ admissible to $([\WCC_1],[\WCC_2])$ such that $\rho=\rho_\varphi$. 

\end{enumerate}
\end{defin}
\begin{ex}
Let $E$ be a smooth cubic curve and $P\in E$ be a general point of $E$. It is known that there are four lines $L_j$ $(j=1, \ldots, 4)$ passing through $P$ and tangent to $E$ at $Q_j\not=P$. Let $\WCC=E+L_1+L_2+L_3+L_4$ and consider decompositions $[\WCC]_1=(E;(L_1+L_2), (L_3+L_4))$ and 
$[\WCC]_2=(E;(L_1+L_3), (L_2+L_4))$. The identity map $ {\rm id}:\PP^2\rightarrow\PP^2$  trivially induces an equivalence map of combinatorics $\varphi_{{\rm id}}: \comb(\WCC)\rightarrow\comb(\WCC)$. However $\varphi_{\rm id}$ is not admissible with respect to $([\WCC]_1, [\WCC]_2)$.

On the other hand, let $P_1, P_2\in E$ be distinct general points and let $L_{ij}$ be lines passing through $P_i$ and are tangent to $E$ at $Q_{ij}\not=P_1, P_2$. Let $\WCC=E+L_{11}+L_{12}+L_{21}+L_{22}$ and consider a decomposition $[\WCC]=(E;(L_{11}+L_{12}), (L_{21}+L_{22}))$. Then every equivalence map $\varphi:\comb(\WCC)\rightarrow\comb(\WCC)$ is admissible with respect to $([\WCC], [\WCC])$ by the definition of $\varphi$. 

\end{ex}

We recall the definition of Zariski tuples. 

\begin{defin}\label{def:zariski pair}
Let $\WCC_1,\dots,\WCC_N$ be $N$ plane curves. 
An $N$-tuple $(\WCC_1,\dots,\WCC_N)$ is called a \textit{Zariski $N$-tuple} if 
\begin{enumerate}
	\item $\WCC_1,\dots,\WCC_N$ have the same combinatorics, and 
	\item $(\PP^2,\WCC_i)$ and $(\PP^2,\WCC_j)$ are not homeomorphic for any $i\ne j$.  
\end{enumerate}
In the case where $N=2$, a Zariski $2$-tuple is called a \textit{Zariski pair}.
\end{defin}

\subsection{A subgroup of \texorpdfstring{$\Div(D)$}{Div(D)} and a homomorphism \texorpdfstring{$\phi_{\fo_D}$}{phio}}\label{subsec:hom_phi}

Assume that $D$ is a smooth plane curve of degree $d_0$. We choose a  line $L$ and set
$\fo_D = L|_D$.  
Let
\[
\Div(D)_{d_0} := \{ \fd \in \Div(D) \mid \text{$d_0$ divides $\deg \fd$} \}.
\]
Since $\deg(\fd_1 - \fd_2) = \deg \fd_1 - \deg \fd_2$ for any $\fd_1,\fd_2\in\Div(D)_{d_0}$, $\Div(D)_{d_0}$ is a subgroup of $\Div(D)$. We define
a map $\psi_{\fo_D}  : \Div(D)_{d_0} \to \pic^0(D)$ by
\[
\psi_{\fo_D} : \fd \mapsto \fd - \frac{\deg\fd}{d_0} \fo_D.
\]
It is easy to see that $\psi_{\fo_D}$ is a homomorphism. We say that $\fd \in \Div(D)_{d_0}$ is a \emph{$\nu$-torsion divisor}
with respect to $\fo_D$ 
if $\psi_{\fo_D}(l\fd) \neq 0$ ($1 \le l \le \nu-1$)  and $\psi_{\fo_D}(\nu\fd) = 0$. 

Let $\fd \in \Div(D)_{d_0}$ be 
an effective divisor which is a $\nu$-torsion with
respect to $\fo_D$. Then there exists  $\varphi \in \CC(D)\setminus\{0\}$, where $\CC(D)$ is the rational function field of $D$,
such that
\[
(\varphi) = \nu\left(\fd - \frac{\deg \fd}{d_0}\fo_D\right).
\]
This means that there exists $s \in H^0(D, \mcO(\frac{\nu\deg\fd}{d_0}\fo_D))$ such that $\nu\fd$ is the effective divisor~$(s)$.  
Let us consider the short exact sequence of sheaves on~$\PP^2$:
\[
0\to\mcO_{\PP^2}\left(\frac{\nu\deg \fd}{d_0}L - D\right)\to
\mcO_{\PP^2}\left(\frac{\nu\deg \fd}{d_0}L\right)\to i_*\mcO_D\left(\frac{\nu\deg \fd}{d_0} \fo_D\right)\to 0,
\]
where $i:D\hookrightarrow\PP^2$ is the inclusion.
As $H^1(\PP^2, \mcO(\frac{\nu\deg \fd}{d_0} L- D) = 0$, 
the map
\[
H^0\left(
\PP^2, \mcO
\left(\frac{\nu\deg \fd}{d_0} L\right)
\right) \to  
H^0\left(D, \mcO\left(\frac{\nu\deg \fd}{d_0} \fo_D\right)\right)
\]
is surjective. This means that there exists a homogeneous polynomial $g$,
of degree $\frac{\nu\deg \fd}{d_0}$, such that 
the plane curve $C$ given by $g=0$ defines $\nu\fd$, i.e., $C|_D = \nu\fd$.

The curve $C$ and the divisor $\fd$ can be considered as a geometric description of a $\nu$-torsion with respect to $\fo_D$.
Note that $\fd$ itself is not given by any plane curve (if $\nu>1$).

\begin{rem}
In \cite{cubic_torsion}, we consider the case where the smooth curve $D$ has a maximal flex point $P$, and discuss a similar argument taking $\fo_D$ as the maximal flex point~$P$.
\end{rem}

\section{Splitting numbers and proofs of main results}\label{sec:splitting number}

In this section we translate the torsion properties 
into topological properties. 
Consequently, we prove  the main results presented in Section \ref{sec:main}.
We first define the splitting number of an irreducible curve for a cyclic cover. 

Let $\phi:X\to\PP^2$ be a cyclic cover of degree $n$ branched along 
a projective curve $\WCB$ given by a surjection $\theta:\pi_1(\PP^2\setminus \WCB)\twoheadrightarrow\ZZ_n$, 
where $\ZZ_n$ is the cyclic group of order $n$. 
Let $\WCB_\theta$ be the following divisor on $\PP^2$
\[ \WCB_\theta:=\sum_{i=1}^{n-1} i\,\mcB_{\theta,i}, \]
where $\mcB_{\theta,i}$ is the sum of irreducible components of $\WCB$ whose meridians are sent to $[i]\in\ZZ_n$ by $\theta$. 
Then we call $\phi:X\to\PP^2$ the \textit{$\ZZ_n$-cover} (or \textit{$n$-cyclic cover}) \textit{of type $\WCB_\theta$}. 
Note that the degree of $\WCB_\theta$ is divisible by $n$ since $\phi$ is of degree $n$. 

\begin{rem}\label{rem:meridians}
Let us recall the definition of a meridian. Let $\mcC$ be a projective plane curve with irreducible components $C_i$.
A \emph{meridian} of $C_i$ in $\pi_1(\PP^2\setminus \mcC)$ is 
a free homotopy class  defined as follows. Let $p_0$ be a base point
and let $p_i\in C_i$ be a smooth point of $\mcC$. Pick  a smooth holomorphic $2$-disk $\Delta_i$ such that $\mcC\cap\Delta_i=\{p_i\}$
and the intersection is transversal. Let $p_i'\in\partial\Delta_i$.
A meridian is the homotopy class of the composition of an arbitrary path $\alpha$ from 
$p_0$ to $p'_i$ in $\PP^2\setminus \mcC$ with $\partial\Delta_i$ (counterclockwise) and $\alpha^{-1}$. As for free homotopy classes
the base point is not important and it can be chosen as $p'_i$.
It is not hard to prove that with other choices the free homotopy
class is not changed (the key is that $C_i\setminus\sing \mcC$ is connected). An \emph{antimeridian} is constructed in the same way
by running $\partial\Delta_i$ clockwise (getting the free homotopy
class of the inverses of the meridians).
\end{rem}

\begin{defin}[\cite{shirane2016}]
Let $D$ be an irreducible curve which is not a component of $\WCB$. The \emph{splitting
number} $s_{\phi}(D)$ of $D$ with respect to~$\phi$ is the number of irreducible
components of~$\phi^*D$. 
\end{defin}

Note that $0\leq n(r-\lfloor r\rfloor)<n$ for any $r\in\RR$, and $n(\frac{l}{n}-\lfloor\frac{l}{n}\rfloor)\equiv l\pmod{n}$ for any $l\in\ZZ$, where $\lfloor r\rfloor$ is the maximal integer not beyond $r\in\RR$. 
We consider splitting numbers of $D$ for certain $\ZZ_n$-covers. 
The following lemma enable us to reduce the number of $\ZZ_n$-covers which should be considered. 

\begin{lem}\label{rem: inverse}
Let $\WCB:=\sum_{j=1}^{n-1}\mcB_j$ be a plane curve. 
For $l\in\ZZ$ with $\gcd(l,n)=1$, put the divisor $\WCB^{(l)}$ as  
\begin{align*}
	\WCB^{(l)} &:=\sum_{i=1}^{n-1}n\left(\frac{il}{n}-\left\lfloor\frac{il}{n}\right\rfloor\right)\mcB_{i}.
\end{align*}
The degree $\deg\WCB^{(l_1)}$ is divisible by $n$ for some $l_1\in\ZZ$ with $\gcd(l_1,n)=1$ if and only if $\deg\WCB^{(l)}$ is divisible by $n$ for any $l\in\ZZ$ with $\gcd(l,n)=1$. 
Moreover, for any irreducible curve $D$ which is not a component of $\WCB$, the splitting numbers $s_{\phi^{(l_1)}}(D)$ and $s_{\phi^{(l_2)}}(D)$ coincide for any $l_1,l_2\in\ZZ$ with $\gcd(l_i,n)=1$ $(i=1,2)$, 
\[ s_{\phi^{(l_1)}}(D)=s_{\phi^{(l_2)}}(D), \]
where $\phi^{(l)}:X\to\PP^2$ is the $\ZZ_n$-cover of type $\WCB^{(l)}$. 
\end{lem}
\begin{proof}
Let $l_1\in\ZZ$ be an integer with $\gcd(l_1,n)=1$. 
Suppose that $\deg\WCB^{(l_1)}$ is divisible by $n$. 
Then there is a surjection $\theta^{(l_1)}:\pi(\PP^2\setminus\WCB)\to\ZZ_n$ given by 
\[ \theta^{(l_1)}(\merid_i)=
\left[ il_1-n\left\lfloor\frac{il_1}{n}\right\rfloor \right]=[il_1]\in\ZZ_n \] 
for any meridian $\merid_i$ of components of $\mcB_{i}$. 
Note that, since $\gcd(l_1,n)=1$, there exists $[l_1']\in\ZZ_n$ such that $[l_1l_1']=[1]\in\ZZ_n$, and the multiplication by $l_1'l$ gives an isomorphism $m_{l_1'l}:\ZZ_n\to\ZZ_n$ for any $l\in\ZZ$ with $\gcd(l,n)=1$. 
Hence we obtain the surjection $\theta^{(l_1)}_l:=m_{l_1'l}{\circ}\theta^{(l)}:\pi_1(\PP^2\setminus\WCB)\to\ZZ_n$ given by  
\[ \theta_l^{(l_1)}(\fm_i)=
[l'_1l{\cdot}il_1]=[il]\in\ZZ_n. \]
Moreover, $\theta^{(l_1)}_l$ defines the $\ZZ_n$-cover $\phi^{(l)}$ of type $\WCB^{(l)}$, and $\deg\WCB^{(l)}$ is divisible by $n$. 
Therefore, there exists a homeomorphism $\tilde{h}_{l_1,l}:X^{(l_1)}\setminus(\phi^{(l_1)})^{-1}(\WCB)\to X^{(l)}\setminus(\phi^{(l)})^{-1}(\WCB)$ with $\phi^{(l_1)}=\phi^{(l)}\circ \tilde{h}$ over $\PP^2\setminus\WCB$ by \cite[Proposition~1.3]{shirane2016}. 
In particular, $s_{\phi^{(l_1)}}(D)=s_{\phi^{(l)}}(D)$. 
\end{proof}

We next show a relation between the splitting numbers and the topology of a plane curve $\WCC$ of the following form;
\[
	\WCC :=D+\WCB \qquad
	\Bigg(\WCB :=\sum_{j=1}^k \mcC_j\Bigg), 
\]
where $D$ is a smooth curve of degree $d_0$, and $\mcC_j$ is a curve of degree $d_j$ for each $j=1,\dots,k$. 
Let $n\geq2$ be an integer, and put $\bm{d}:=(d_1,\dots,d_k)$. 
Let 
\[ 
\Theta_{\bm{d}}^n:=\left\{ (a_1,\dots,a_k)\in\ZZ^k \ \middle|\ 
\begin{array}{l}
a_1d_1+\dots+a_k d_k\equiv 0\pmod{n} \\
\gcd(a_1,\dots,a_k,n)=1
\end{array}
 \right\}. 
 \] 
Note that, if $\bm{a}\in\Theta_{\bm{d}}^n$, there is a surjection $\theta:\pi_1(\PP^2\setminus\WCB)\twoheadrightarrow\ZZ_n$ given by $\theta(\merid_j)=[a_j]$ for meridians $\merid_j$ of components of $\mcC_j$ ($j=1,\dots,k$). 
Moreover, this set essentially measures $\ZZ_n$-covers of $\PP^2$ whose ramification
locus is contained in $\WCB$
(if the coordinates are considered $\bmod\ n$). 

\begin{prop}\label{cor:meridian}
Let $\mcC_i=C_{i,1}+\dots+C_{i,k}$, $i=1,2$ be two curves such that there 
exists a homeomorphism $h:\PP^2\to\PP^2$, where $h(C_{1,j})=C_{2,j}$ for $j=1,\dots,k$.
Then, a meridian of $C_{1,j}$ is sent to either a meridian or an antimeridian of $C_{2,j}$.
\end{prop}

\begin{proof}
We are going to use an alternative description of meridians and antimeridians.
Let us consider a closed polydisk (in some analytic coordinates $u,v$)
$P_i$ centered at $p_i$ such that $v=0$ is 
$P_i\cap \mcC=P_i\cap C_i$. Then $\pi_1(P_i\setminus \mcC)=H_1(P_i\setminus \mcC;\ZZ)\cong\ZZ$ since $P_i\setminus \mcC$ has the homotopy type of a circle. Meridians and antimeridians of $C_i$ in $P_i$ are generators of this group.

%
Let $\mcC_1,\mcC_2\subset\PP^2$ and 
$h:(\PP^2,\mcC_1)\to(\PP^2,\mcC_2)$ be as in the statement. 
Let $h_\ast:H_2(\PP^2;\ZZ)\to H_2(\PP^2;\ZZ)$ be the isomorphism induced by $h$. 
Recall that $\PP^2$ and the irreducible components of $\mcC_1,\mcC_2$
have a natural orientation inherited by the complex structure.
By abuse of notation, we denote the class in $H_2(\PP^2;\ZZ)$ containing an irreducible curve $C$ by the same symbol $C$. 
Using the intersection form, we deduce that $h$ must respect
the orientation of~$\PP^2$.

Let us choose a meridian $\gamma_{1,i}$ of $C_{1,i}$ in some polydisk
$P_{1,i}$ centered at a point $p_{1i}\in C_{1,i}$ such that $p_{1,i}$ is a smooth point of $\mcC_1$. 
Let $p_{2,i}:=\Phi(p_{1,i})$
a smooth point of $\mcC_2$ in $C_{2,i}$. Taking $P_{1,i}$ small enough
we can find a polydisk centered at $p_{2,i}$ as in the previous discussion and such that $h(P_{1,i})\subset P_{2,i}$. Since 
$h(P_{1,i})$ is a neighborhood of $p_{2,i}$, it is possible to
find small enough polydisks $\hat{P}_{1,i}$ (centered at $p_{1,i}$)
and $\hat{P}_{2,i}$ (centered at $p_{2,i}$) such that 
\[
\hat{P}_{2,i}\subset h(\hat{P}_{2,i})\subset h(P_{1,i})\subset
P_{2,i}
\]
and the inclusions $\hat{P}_{1,i}\setminus C_{1,i}\subset {P}_{1,i}\setminus C_i$ and $\hat{P}_{2,i}\setminus C_{2,i}\subset {P}_{2,i}\setminus C_{2,i}$
are homotopy equivalences. Hence, we deal with the following
commutative diagram.
\[
\begin{tikzcd}
H_1(\hat{P}_{i,i}\setminus C_{1,i};\ZZ)\ar[d,"\cong"]&
H_1(\hat{P}_{2,i}\setminus C_{2,i};\ZZ)\ar[l,"h^{-1}_*" above]\ar[d,"\cong"]\\
H_1({P}_{1,i}\setminus C_{1,i};\ZZ)\ar[r,"h_*"]&
H_1({P}_{2,i}\setminus C_{2,i};\ZZ)
\end{tikzcd}
\]
From this diagram we deduce that $h_*(\gamma_{1,i})$
is a generator of $H_1(\hat{P}_{2,i}\setminus C_{2,i};\ZZ)$.
Hence, the image is ${\gamma_{2,i}}^\varepsilon$,
$\varepsilon=\pm1$, where $\gamma_{2,i}$ is a meridian of of $C_{2,i}$.
\end{proof}

\begin{rem}\label{rem:meridian}
Let $\WCB_i=\sum_{j=1}^kC_{ij}$ be the irreducible decomposition of a plane curve $\WCB_i$ for $i=1,2$, and let $\merid_{ij}$ be a meridian of $C_{ij}$. 
If $h:\PP^2\to\PP^2$ is a homeomorphism with $h(C_{1j})=C_{2j}$ for $j=1,\dots,k$, then there exists $\varepsilon\in\{\pm1\}$ such that $h_\ast(\merid_{1j})=\merid_{2j}^\varepsilon$ for all $j=1,\dots,k$. 

Indeed,  by Proposition~\ref{cor:meridian} we know that $h_\ast(\merid_{1j})=\merid_{2j}^{\varepsilon_j}$, $\varepsilon_j=\pm1$.
By Poincar{\'e}-Lefschetz duality we have that
\[
H_1(\PP^2\setminus\WCB_i)=H^3(\PP^2,\WCB_i)=\coker(H^2(\PP^2)\to H^2(\WCB_i))
\]
The homotopy classes of the meridians are sent to the dual of the classes of $C_{ij}$ in $H_2(\WCB_i)$ and the linking numbers
are determined by the signs, hence 
$h_\ast(C_{1j})=\epsilon_j C_{2j}$ as $h_\ast:H_2(\PP^2;\ZZ)\to H_2(\PP^2;\ZZ)$ is induced by $h$. Then for $d_j:=\deg C_{ij}$,
\[ d_{j}d_{j'}=C_{1j}\cdot C_{1j'}=h_\ast(C_{1j})\cdot h_\ast(C_{1j'})=\epsilon_{j}\epsilon_{j'}C_{2j}\cdot C_{2j'}=\epsilon_j\epsilon_{j'}d_jd_{j'} \Longrightarrow \epsilon_j\epsilon_{j'}=1; \]
hence $\epsilon_j=\epsilon_{j'}=\pm1$ for any $j,j'=1,\dots,k$, and $h_\ast(\merid_{1j})=\merid_{2j}^\epsilon$ for some $\epsilon\in\{\pm1\}$ since a canonical orientation of $\merid_{ij}$ is determined by orientations of $\PP^2$ and $C_{ij}$. 
\end{rem}

Put $[\WCC]:=(D;\mcC_1,\dots,\mcC_k)$. 
For $\bm{a}=(a_1,\dots,a_k)\in\Theta_{\bm{d}}^n$, 
let 
\[ \WCB_{\bm{a}}[\WCC]:=\left(a_1-n\left\lfloor\frac{a_1}{n}\right\rfloor\right)\mcC_1+\dots+\left(a_k-n\left\lfloor\frac{a_k}{n}\right\rfloor\right)\mcC_k. \] 
Put $\WCB_{\bm{a}}:=\WCB_{\bm{a}}[\WCC]$. Then we have $\deg\WCB_{\bm{a}}\equiv0\pmod{n}$. 
For $\bm{a}\in\Theta_{\bm{d}}^n$, 
we define a map $\Phi_{[\WCC]}^n:\Theta_{\bm{d}}^n\to\ZZ$ by 
\[ \Phi_{[\WCC]}^n(\bm{a}):=s_{\phi_{\bm{a}}}(D), \]
where $\phi_{\bm{a}}:X_{\bm{a}}\to\PP^2$ is the $\ZZ_n$-cover of type $\WCB_{\bm{a}}$. 
By Lemma~\ref{rem: inverse}, for any $\bm{a}\in\Theta_{\bm{d}}^n$ and any $l\in\ZZ$ with 
$\gcd(l,n)=1$, we obtain 
\[ \Phi_{[\WCC]}^n(\bm{a})=\Phi_{[\WCC]}^n(l\bm{a}). \]
Since an admissible permutation $\rho$ to $([\WCC],[\WCC])$ satisfies $d_{\rho(j)}=d_{j}$ by definition, 
we have 
\begin{align*} 
a_1d_1+\dots+a_kd_k&=
a_1d_{\rho(1)}+\dots+a_kd_{\rho(k)} \\
&=
a_{\rho^{-1}(1)}d_1+\dots+a_{\rho^{-1}(k)}d_k. 
\end{align*}
Hence $\rho$ acts on $\Theta_{\bm{d}}^n$ by $\rho(a_1,\dots,a_k)=(a_{\rho^{-1}(1)},\dots,a_{\rho^{-1}(k)})$.

\begin{prop}\label{prop: splitting map}
Let $\WCB_i$ and $\WCC_i$ $(i=1,2)$ be plane curves
\begin{align*}
   \WCB_i&:=\sum_{j=1}^k \mcC_{i,j}, &  \WCC_i &:=D_i+\WCB_i 
\end{align*}
where $D_i$ is irreducible of degree $d_0$, $\mcC_{i,j}$ are curves of degree $d_j$ for each $i=1,2$. 
If there exists a homeomorphism $h:\PP^2\to\PP^2$ such that $h(D_1)=D_2$ and $h(\mcC_{1,j})=\mcC_{2,\rho(j)}$ $(j=1,\dots,k)$ for a permutation $\rho$ of $k$ letters, 
then $\rho$ is admissible to $([\WCC_1],[\WCC_2])$, and the equation 
\[ \Phi_{[\WCC_1]}^n (\bm{a})=\Phi_{[\WCC_2]}^n\big(\rho(\bm{a})\big) \]
holds for any $n\geq 2$ and any $\bm{a}\in\Theta_{\bm{d}}^n$, where $[\WCC_i]=(D_i;\mcC_{i,1},\dots,\mcC_{i,k})$. 
\end{prop}
\begin{proof}
Since the permutation $\rho$ satisfies $h(\mcC_{1,j})=\mcC_{2,\rho(j)}$, 
$\rho$ is admissible to the pair $([\WCC_1],[\WCC_2])$. 
In particular, $\rho$ acts on $\Theta_{\bm{d}}^n$ by 
\[ \rho(a_1,\dots,a_k)=\left(a_{\rho^{-1}(1)},\dots,a_{\rho^{-1}(k)}\right). \]
Fix $n\geq 2$. 
For $\bm{a}\in\Theta_{\bm{d}}^n$, 
let $\phi_{i,\bm{a}}:X_{i,\bm{a}}\to\PP^2$ be the $\ZZ_n$-cover of type $\WCB_{i,\bm{a}}:=\WCB_{\bm{a}}[\WCC_i]$. 
Let $h_\ast:H_2(\PP^2;\ZZ)\to H_2(\PP^2;\ZZ)$ be the isomorphism induced by $h$. 
By Remark~\ref{rem:meridian}, 
\begin{align*}
	h_\ast(a_1\mcC_{1,1}+\dots+a_k\mcC_{1,k})&=\varepsilon\left(a_1\mcC_{2,\rho(1)}+\dots+a_k\mcC_{2,\rho(k)} \right)
	\\
	&=\varepsilon\left(a_{\rho^{-1}(1)}\mcC_{2,1}+\dots+a_{\rho^{-1}(k)}\mcC_{2,k}\right)
\end{align*}
for $\varepsilon=\pm1$, where $\mcC_{i,j}$ is regarded as the class represented by $\mcC_{i,j}$ with the orientation induced by the complex structure for each $i,j$. 
Hence we have $s_{\phi_{1,\bm{a}}}(D_1)=s_{\phi_{2,\rho(\bm{a})}}(D_2)$ by \cite[Proposition~1.3]{shirane2016} and Lemma~\ref{rem: inverse}. 
Therefore, we obtain $\Phi_{[\WCC_1]}^n=\Phi_{[\WCC_2]}^n\circ\rho$. 
\end{proof}

Again we consider the plane curve $\WCC:=D+\WCB$ as above. 
Put $n:=n_{[\WCC]}$.  
Let $\phi_{\bm{a}}:X_{\bm{a}}\to\PP^2$ be the $\ZZ_n$-cover 
of type~$\WCB_{\bm{a}}:=\WCB_{\bm{a}}[\WCC]$ for $\bm{a}\in\Theta_{\bm{d}}^n$. 
We consider the following divisors on $D$: $\fo_D:=L|_D$ for some line~$L$,
and:
\[ 
\fd_{\bm{a}}[\WCC]:=\sum_{P\in\WCB_{\bm{a}}\cap D}\frac{(\WCB_{\bm{a}}\cdot D)_P}{n}P,\qquad 
\left(\deg\fd_{\bm{a}}[\WCC]=\frac{d_0}{n}\sum_{j=1}^k a_j d_j\right). 
\]
Then $\fd_{\bm{a}}:=\fd_{\bm{a}}[\WCC]$ is an element of $\Div(D)_{d_0}$ since $\deg\WCB_{\bm{a}}\equiv0\pmod{n}$ by definition of $\Theta_{\bm{d}}^n$. 
Let $\nu_{\bm{a}}:=\nu_{\bm{a}}[\WCC]$ be the positive integer such that $\fd_{\bm{a}}$ is a $\nu_{\bm{a}}$-torsion divisor with respect to $\fo_D$. 

\begin{prop}\label{prop:key}
In the above setup, 
the following equation holds: 
\[ s_{\phi_{\bm{a}}}(D)=\frac{n}{\nu_{\bm{a}}}\,.\]
\end{prop}
\begin{proof}
Since $\nu_{\bm{a}}(\fd_{\bm{a}}-\frac{\deg\WCB_{\bm{a}}}{n}\fo)\sim 0$ by the hypothesis, there is a divisor $C$ on $\PP^2$ such that $C|_D=\nu_{\bm{a}}\fd_{\bm{a}}$ as divisors on $D$ by the arguments of Subsection \ref{subsec:hom_phi} . As $\nu_{\bm{a}}$ is the minimal positive integer 
with this property, by \cite[Theorem~2.1]{ben-shi}, we obtain $s_{\phi_{\bm{a}}}(D)=\frac{n}{\nu_{\bm{a}}}$. 
\end{proof}

We next prove a corollary of Propositions~\ref{prop: splitting map} and \ref{prop:key}, which gives a method of distinguishing embedded topology of plane curves. 
Let $\WCC_i:=D_i+\sum_{j=1}^k \mcC_{i,j}$ $(i=1,2)$ be two plane curves, 
where $D_i$ are smooth, and $\mcC_{i,j}$ are of degree~$d_i$, and put $[\WCC_i]:=(D_i;\mcC_{i,1},\dots,\mcC_{i,k})$. 
Assume that there is an equivalence map $\varphi:\comb(\WCC_1)\to\comb(\WCC_2)$ admissible to $([\WCC_1],[\WCC_2])$. 
Then we can put
\[
n:=n_{[\WCC_1]}=n_{[\WCC_2]}. 
\]
For $\bm{a}\in\Theta_{\bm{d}}^n$, 
let $\phi_{i,\bm{a}}:X_{i,\bm{a}}\to\PP^2$ be the $\ZZ_n$-cover of type $\WCB_{i,\bm{a}}:=\WCB_{\bm{a}}[\WCC_i]$. 
The next corollary follows from Proposition~\ref{prop: splitting map} and \ref{prop:key}. 

\begin{cor}\label{cor:key}
Let $\WCC_1$ and $\WCC_2$ be two plane curves as above. 
Let $\rho$ be an admissible permutation with respect to $\bm{d}:=(d_1,\dots,d_k)$, 
If there is an element $\bm{a}\in\Theta_{\bm{d}}^n$ such that $\nu_{\bm{a}}[\WCC_1]\ne\nu_{\rho(\bm{a})}[\WCC_2]$, then there exists no homeomorphism $h:\PP^2\to\PP^2$ such that $h(D_1)=D_2$ and $h(\mcC_{1j})=\mcC_{2\rho(j)}$. 
\end{cor}


The main results in Section~\ref{sec:main} follow from Corollary~\ref{cor:key}.

\begin{proof}[Proof of Proposition{\rm~\ref{prop:n-torsion}}]
Put $\ft_{i,j}:=\ft_j[\WCC_i]$. 
Suppose that $\ker\tau_{[\WCC_1]}\not\subset\ker\tau_{[\WCC_2]}\circ\rho$. 
Then there is an element $(a_1^\prime,\dots,a_k^\prime)\in\ZZ^{\oplus k}$ such that 
\[
a_{1}^\prime\ft_{1,1}+\dots+a_{k}^\prime\ft_{1,k}=0\in G[\WCC_1], \qquad 
a_{\rho^{-1}(1)}^\prime\ft_{2,1}+\dots+a_{\rho^{-1}(k)}^\prime\ft_{2,k}\neq 0\in G[\WCC_2]. 
\]
Since $n\ft_{i,j}=0$, we may assume that $0\leq a_j^\prime\leq n-1$ for each $j=1,\dots,k$. 
Let
\[
\bm{a}:=(a_1,\dots,a_k)=\frac{(a'_1,\dots,a'_k)}{\gcd(a'_1,\dots,a'_k)}.
\]
Since $n$ is a divisor of $\gcd(d_1,\dots,d_k)$, we have $\bm{a}:=(a_1,\dots,a_k)\in\Theta_{\bm{d}}^n$. 
Moreover, we have 
\[ \ord\left(a_{1}\ft_{1,1}+\dots+a_{k}\ft_{1,k}\right)\ne\ord\left(a_{\rho^{-1}(1)}\ft_{2,1}+\dots+a_{\rho^{-1}(k)}\ft_{2,k}\right). \]
Since the above orders correspond to $\nu_{\bm{a}}[\WCC_1]$ and $\nu_{\rho(\bm{a})}[\WCC_2]$ in Corollary~\ref{cor:key}, there exists no homeomorphism $h:\PP^2\to\PP^2$ such that $h(D_1)=D_2$ and $h(\mcC_{1,j})=\mcC_{2,\rho(j)}$ by Corollary~\ref{cor:key}. 
In the case of $\ker\tau_{[\WCC_1]}\not\supset\ker\tau_{[\WCC_2]}\circ\rho$, the assertion can be proved by the same way. 
\end{proof}

We next prove Theorem~\ref{thm:n-torsion}.

\begin{proof}[Proof of Theorem{\rm~\ref{thm:n-torsion}}]
Suppose there exists a homeomorphism of topological pairs $h\!:(\PP^2,\WCC_1)\to(\PP^2,\WCC_2)$. 
By the assumption, we have $h(D_1)=D_2$ and $h(\mcC_{1,j})=\mcC_{2,\rho(j)}$ for some admissible permutation $\rho$ with respect to $\bm{d}$. 
This is a contradiction to Proposition~\ref{prop:n-torsion}. 
\end{proof}

We prove Corollary~\ref{cor:n-torsion}. 

\begin{proof}[Proof of Corollary{\rm~\ref{cor:n-torsion}}]
Parts 
\ref{cor:n-torsion-iii} and \ref{cor:n-torsion-i} are clear. 
Let us prove \ref{cor:n-torsion-ii}. 
Suppose that there is a homeomorphism $h : (\PP^2, \WCC_1) \to (\PP^2, \WCC_2)$. 
By the assumption, there exists a permutation $\rho$ such that $\varphi_{h,\irr}(\irr_{\mcC_{1,j}})=\irr_{\mcC_{2,\rho(j)}}$ for $j=1,\dots,k$, i.e., $h(\mcC_{1,j})=\mcC_{2,\rho(j)}$. 
In particular, $h$ gives a homeomorphism $(\PP^2,D_1+\mcC_{1,j})\to(\PP^2,D_2+\mcC_{2,\rho(j)})$ for each $j$. 
However, there is $j_0$ such that $\ord(\ft_{1,j_0})\ne\ord(\ft_{2,\rho(j_0)})$ by the assumption. 
By \ref{cor:n-torsion-i}, the homeomorphism $h$ as above does not exist. 
\end{proof}

\section{Remarks on previous examples}\label{sec:remarks}

We here explain two previous examples where $D$ is a smooth cubic $E$ from our new viewpoint in this article.
For a smooth cubic $E$, there is a classically well-known description on $\pic^0(E)$. In 
\cite{cubic_torsion}, we reformulate our setting in this case and discuss our problem in detail.

\subsection{\texorpdfstring{$3$}{3}-Artal arrangements}

A $k$-Artal arrangement is a reduced plane curve whose irreducible  components are a smooth cubic curve and  $k$ of its tangent lines at inflection points.
A Zariski pair for $3$-Artal arrangements is given in \cite{artal94} and those for $k$-Artal arrangements are
considered in \cite{bgst}. We here
consider $3$-Artal arrangements based on our approach. Note that this is essentially considered in \cite{bgst} from
the viewpoint of {\it splitting numbers}. 

Let $E$ be a smooth cubic curve and let $L_i$ $(i = 1,2, 3)$ be a tangent lines at inflection points $P_i$ ($i = 1, 2, 3$). Put $\triangle := L_1 + L_2 +L_3$, 
$\WCC:=E+\triangle$ and $[\WCC]:=(E;\triangle)$.  
We now choose a line $L$ and put $\fo_E := L|_E$. In this case, 
\[
\fd_{\triangle}:= \fd_1[\WCC] = P_1 + P_2 +P_3, \qquad 
\ft_{\triangle}:=\ft_1[\WCC] = \fd_\triangle - \fo_E. 
\]
If $P_1, P_2$ and $P_3$ are collinear, i.e., there exists a line  $L_{\triangle}$ such that
$L_{\triangle}|_E = P_1 + P_2 +P_3$, we have $\ft_\triangle = 0$ in $\pic^0(E)$. 
On the other hand,
if $P_1, P_2$ and $P_3$ are not collinear, $\ft_\triangle$ is a $3$ torsion as an element of $\pic^0(E)$. 
Hence
we choose $\triangle_i:=L_{i1}+L_{i2}+L_{i3}$ ($i=1,2$), where $L_{ij}$ are inflectional tangents of $E$, such that 
\begin{enumerate}
	\item both of $\triangle_i$ ($i = 1, 2)$ are not concurrent three lines, 
	\item the three inflection points for $\triangle_1$ are collinear, and 
	\item the three inflection points for $\triangle_2$ are not collinear. 
\end{enumerate}
Then $\WCC_i:=E+\triangle_{i}$ ($i=1,2$) have the same combinatorics, and 
any equivalence map $\varphi:\comb(\WCC_1)\to\comb(\WCC_2)$ is admissible to $([\WCC_1],[\WCC_2])$ since $\deg\varphi_{\irr}(C_1)=\deg C_1$ for each $C_1\in\irr_{\WCC_1}$. 
Therefore $(\WCC_1, \WCC_2)$ is a
Zariski pair by Corollary~\ref{cor:n-torsion}~\ref{cor:n-torsion-i}.

\subsection{A smooth cubic and \texorpdfstring{$4$}{4} tangent lines}
In \cite{bannai-tokunaga2019}, we study Zariski tuples for a smooth cubic $E$ and its tangent lines. We distinguish
the embedded topology by counting the number of basic subarrangements consisting of $E$ and  $4$ of its tangent
lines as follows:

Choose two distinct points $P_1$ and $P_2$ of $E$. For each $P_j$ $(j=1,2)$, we choose two tangent lines $L_{P_j, 1}$ and
$L_{P_j, 2}$ tangent at $Q_{j, 1}$ and $Q_{j, 2}$, respectively. Put
\[
\Lambda_j = L_{P_j, 1} + L_{P_j, 2}.
\]
We consider plane curves of the form $\WCC:=E+\Lambda_1+\Lambda_2$, and put $[\WCC]:=(E;\Lambda_1,\Lambda_2)$. 
Then we have
\[
\fd_{\Lambda_i} := \fd_j[\WCC] = P_j + Q_{j, 1} + Q_{j, 2}.
\]
Now we choose a line $L$ and put $\fo_E := L|_E$. Since $P_j, Q_{j, 1}, Q_{j, 2}$ cannot be
collinear, 
\[
\ft_{\Lambda_j} := \ft_j[\WCC]
= P_j + Q_{j, 1} + Q_{j, 2} - \fo_E \not\sim 0,
\]
Hence $\ft_{\Lambda_j}\in\pic^0(E)[2]$ as $2\ft_{\Lambda_j} = 2\fd_{\Lambda_j}-2\fo_E = 0$ in $\pic^0(E)$. 
Note that $\ft_{\Lambda_j}$ 
depends on the choice of $L_{P_j, k}$ $(k = 1, 2)$ as
there exist $4$ tangent lines passing through $P_j$.  Choose $\{\Lambda_1, \Lambda_2\}$ and
$\{\Lambda'_1, \Lambda'_2\}$ such that $\ft_{\Lambda_j}$ and $\ft_{\Lambda_j'}$ as elements of $\pic^0(E)$ satisfy 
\begin{center}
$\ft_{\Lambda_1} = \ft_{\Lambda_2}$ \quad and \quad $\ft_{\Lambda'_1} \neq \ft_{\Lambda'_2}$ \quad
in \quad $\pic^0(E)[2]\cong\ZZ_2\oplus\ZZ_2$. 
\end{center}
Put 
\begin{align*}
	\WCC&:=E+\Lambda_1+\Lambda_2, \qquad 
	 [\WCC]:=(E;\Lambda_1,\Lambda_2), \\
	\WCC'&:=E+\Lambda_1'+\Lambda_2', \qquad
	[\WCC']:=(E;\Lambda_1',\Lambda_2'). 
\end{align*}
Then  we obtain $G[\WCC]\cong\ZZ_2$ and $G[\WCC']\cong\ZZ_2\oplus\ZZ_2$, and $\WCC$ and $\WCC'$ have the same combinatorics. 

Let $\varphi:\comb(\WCC)\to\comb(\WCC')$ be any equivalence map. 
We prove that $\varphi$ is admissible to $([\WCC],[\WCC'])$. 
Put $\Lambda_j':=L_{P_j',1}+L_{P_j',2}$, and $E\cap L_{P_j',k}=\{P_j',Q'_{j,k}\}$, where $Q'_{j,k}$ is the tangent point of $E$ and $L_{P'_{j},k}$. 
Note that 
\begin{align*} 
\sing_{\WCC}=\{P_{1},P_2,Q_{1,1},Q_{1,2},Q_{2,1},Q_{2,2}\}, \\ 
\sing_{\WCC'}=\{P_{1}',P_2',Q_{1,1}',Q_{1,2}',Q_{2,1}',Q_{2,2}'\}. 
\end{align*}
Since $\deg E>1$ for any $j,k$, we have $\varphi_{\irr}(E)=E$. 
Since $\sharp\WCC(P_j)=\sharp\WCC'(P_j')=3$ and $\sharp\WCC(Q_{j,k})=\sharp\WCC'(Q_{j,k}')=2$ for each $j,k$, $\varphi_{\sing}(P_j)=P'_{\rho(j)}$ ($j=1,2$) for a certain permutation $\rho$ of two letters. 
Moreover, there is a permutation $\rho'$ of two letters such that $\varphi_{\sing}(Q_{j,k})=Q_{\rho(j),\rho'(k)}'$ since $\varphi$ satisfies $\beta_{\mcC_2,\varphi_{\sing}(P)}\circ\varphi_P=\varphi_{\irr\circ\beta_{\mcC_1,P}}$. 
Hence we have $\varphi_{\irr}(\irr_{\Lambda_j})=\irr_{\Lambda_{\rho(j)}'}$, and $\varphi$ is admissible to $([\WCC],[\WCC'])$. 
Therefore, 
by Corollary~\ref{cor:n-torsion}~\ref{cor:n-torsion-iii}, $(\WCC,\WCC')$ is a Zariski pair.

\section{Generalizations of Shimada's curves  and \texorpdfstring{$n$}{n}-tor\-sion divisors}\label{sec:example}

In this section, we construct certain plane curves consisting of two smooth curves, which are generalizations of Shimada's curves in \cite{shimada2003}. 
\begin{defin}
Let $\WCC:=D+C$ be a plane curve, where $D$ and $C$ are smooth curves of degrees $d_0$ and $d_1$ with $d_1\geq d_0$. 
For a divisor $n$ of $d_1$, 
we call $\mcC$ a \emph{plane curve of type $(d_0,d_1;n)$}  if 
the local intersection number $(D\cdot C)_P=n$ for each $P\in C\cap D$. 	
\end{defin}

\begin{rem}
A plane curve of type $(3,d_1;n)$ is called a plane curve of type $(d_1,n)$ by Shimada in \cite{shimada2003}. 
\end{rem}

Let $\WCC:=D+C$ be a plane curve of type $(d_0,d_1;n)$, and put $[\WCC]:=(D;C)$. 
The divisors $\fo_D$ and $\fd:=\fd_1[\WCC]$ on $D$ are defined as follows: 
\begin{align*} 
\fo_D &:=L|_{D}, \quad \fd :=\sum_{P\in C\cap D} P, 
\end{align*}
where $L$ is a general line on $\PP^2$. 
Since $n\fd=C|_{D}$ as a divisor on $D$, we have that $n\fd\sim d_1\fo_D$,
$n(\fd-\frac{d_1}{n}\fo_D)$ vanishes in $\pic^0(D)$ and the order of $\fd-\frac{d_1}{n}\fo_D$
is a divisor of~$n$. 

\begin{defin}\label{def:order_2smooth}
Let $\WCC:=D+C$ be a plane curve of type $(d_0,d_1;n)$. 
We call $\WCC$ a \textit{plane curve of type $(d_0,d_1;n,\nu)$} 
if the torsion order of $\fd-\frac{d_1}{n}\fo_D$ is exactly~$\nu$. 
\end{defin}

By Corollary~\ref{cor:n-torsion}, we obtain the following corollary. 

\begin{cor}\label{cor:example}
Let $\WCC_i:=D_i+C_i$ $(i=1,2)$ be plane curves of type $(d_0,d_1;n,\nu_i)$. 
Assume that $\WCC_1$ and $\WCC_2$ satisfy the following conditions: 
\begin{enumerate}
	\item $\deg D_i=d_0$ and $\deg C_i=d_1$ for $i=1,2$, 
	\item $d_0<d_1$, and
	\item $\nu_1\ne\nu_2$.
\end{enumerate}
Then $(\WCC_1,\WCC_2)$ is a Zariski pair. 
\end{cor}

\begin{proof}
By the definition of plane curves of type $(d_0,d_1;n)$, we have $\sing_{\WCC_i}=D_i\cap C_i$, $\sharp\sing_{\WCC_i}=\frac{d_0d_1}{n}$, and all singularities of $\WCC_i$ have the same topological types. 
Hence the bijections $\varphi_{\irr}:\irr_{\WCC_1}\to\irr_{\WCC_2}$ and $\varphi_{\sing}:\sing_{\WCC_1}\to\sing_{\WCC_2}$ induces an equivalence map $\varphi:\comb(\WCC_1)\to\comb(\WCC_2)$, where  $\varphi_{\sing}$ is a bijective map, and $\varphi_{\irr}$ is given~by 
\begin{align*}
\varphi_{\irr}(D_1):=D_2, \quad \varphi_{\irr}(C_1):=C_2.
\end{align*}
Thus $\WCC_1$ and $\WCC_2$ have the same combinatorics. 
Furthermore, any equivalence map $\varphi:\comb(\WCC_1)\to\comb(\WCC_2)$ satisfies $\varphi_{\irr}(D_1)=D_2$ and $\varphi_{\irr}(C_1)=C_2$ since $d_0<d_1$. 
Therefore $(\WCC_1,\WCC_2)$ is a Zariski pair by Corollary~\ref{cor:n-torsion}~\ref{cor:n-torsion-i} since $\nu_1\ne\nu_2$. 
\end{proof}

The following proposition provides a method for the construction of a new plane curve of type $(d_0',d_1';n',\nu')$ from a given a plane curve $\WCC=D+C$ of type $(d_0,d_1;n,\nu)$. 
Let $f_0\in H^0(\PP^2,\mcO(d_0))$ be a homogeneous polynomial defining $D$. 
From \cite[Theorem~2.7]{shirane2016} we deduce that 
$\nu$ is the minimal number for which there are homogeneous polynomials $g$ and $h$ of degree $d_1-d_0$ and $\frac{d_1}{\mu}$, respectively, such that $C$ is defined by $f_1:=h^\mu+f_0g=0$, where $\mu:=\frac{n}{\nu}$.

\begin{prop}\label{prop:type_2smooth}
Let $\WCC:=D+C$ be a plane curve of type $(d_0, d_1; n,\nu)$, 
where~$D$ is defined by $f_0=0$ and $C$ is defined by $f_1:=h^\mu+f_0g=0$ with $\mu:=\frac{n}{\nu}$. 
Let $k$ be an integer satisfying $kd_0\geq d_1$. 

Then the curve  $B_{D,k}$ defined by $f_0^k+f_1g'=0$ is smooth for a general homogeneous polynomial $g'$ of degree $kd_0-d_1$. 
Moreover, $B_{D,k}+C$ is a plane curve of type 
\[
\begin{cases}
(d_1,kd_0;kn, n)&\text{ if }d_0\equiv0\pmod{n}\text{ and }d_0<d_1,\\
(d_1,kd_0;kn,\nu)&\text{ if }d_0=d_1.
\end{cases}
\]

\end{prop}
\begin{proof}
Let $\Lambda$ be the linear system on $\PP^2$ consisting of divisors defined by $f_0^k+f_1g'$ for $g'\in H^0(\PP^2,\mcO(kd_0-d_1))$. 
The base points of $\Lambda$ is just the $\frac{d_0d_1}{n}$ points of $C\cap D$. 
By Bertini's theorem (see \cite{kleiman74}), a general member of $\Lambda$ is smooth outside~$C\cap D$. 

Moreover, if $g'\in H^0(\PP^2,\mcO(kd_0-d_1))$ does not vanish at each $P\in C\cap D$, $f_1g'=0$ defines a plane curve which is smooth at each $P\in C\cap D$. 
Hence $f_0^k+f_1g'=0$ defines a smooth curve $B_{D,k}$ for a general $g'$. 
It is clear that $C\cap D=C\cap B_{D,k}$. Moreover,
since $C$ and $D$ intersect with multiplicity $n$ at each $P\in C\cap D$, 
we have $(B_{D,k}\cdot C)_P=kn$. 
Thus $\WCC_{D,k}:=B_{D,k}+C$ is a plane curve of type $(d_1,kd_0;kn)$ if $d_0\equiv0\pmod{n}$. 

Let $\fd_{D,k}$ and $\fo_{D,k}$ be the divisors on $C$ for the plane curve $\WCC_{D,k}$ as given before Definition~\ref{def:order_2smooth}. 
Since $B_{D,k}$ is defined by $f_0^k+f_1g'=0$, $n$ is a divisor of the order of $\fd_{D,k}-\frac{d_1}{n}\fo_{D,k}$. 
By the short exact sequence
\[ 
0\to H^0(\PP^2,\mcO(d_0-d_1))\to H^0(\PP^2,\mcO(d_0))\to H^0(C,\mcO(d_0))\to0,
\]
$D$ is the unique divisor on $\PP^2$ satisfying $D|_C=n\fd_{D,k}$ if $d_0<d_1$. 
Since $D$ is reduced, $\mcC_{D,k}$ is of type $(d_1,kd_0;kn,n)$ if $d_0<d_1$. 

Suppose that $d_0=d_1$. 
Then, since $C\ne D$, $g$ is a non-zero complex number, and we have $f=g^{-1}(f_1-h^\mu)$. 
By the minimality of $\nu$, $D$ induces a divisor of order $\nu$ on $\pic^0(C)$ as before Definition~\ref{def:order_2smooth}. 
Since $kD|_C=B_{D,k}|_C$, $\mcC_{D,k}$ is a plane curve of type $(d_1,kd_0;kn,\nu)$. 
\end{proof}

\begin{defin}
In the situation of Proposition~\ref{prop:type_2smooth} we say that the 
plane curve of type $(d_1,kd_0;kn,\nu')$
is \emph{constructed from a curve of type $(d_0,d_1;n,\nu)$ by power of $k$}, and write $(d_0,d_1;n,\nu)\overset{k}{\leadsto}(d_1,kd_0;kn,\nu^\prime)$, where 
\[
\begin{cases}
\nu'=\nu&\text{ if }d_0=d_1,\\
\nu'=n&\text{ if } d_0\equiv0\pmod{n} \mbox{ and } d_0<d_1.
\end{cases}
\]
\end{defin}

\begin{ex}[Plane curves of type $(4,6;6,1)$ and $(4,6;6,2)$]\label{ex:6461_6462}
Since $D+C$ is a plane curve of type $(d_0,d_1;1,1)$ if $C$ and $D$ intersect transversally, there are 
plane curves of type $(d_0,d_1;1,1)$ for any $d_0\leq d_1$. 
We have 
\begin{align*}
&(1,4;1,1)\overset{6}{\leadsto}(4,6;6,1), \\
&(2,2;1,1)\overset{2}{\leadsto}(2,4;2,1)\overset{3}{\leadsto}(4,6;6,2).	
\end{align*}
Therefore, there are plane curves $\mcC_1$ and $\mcC_2$ of types $(4,6;6,1)$ and $(4,6;6,2)$, respectively. 
\end{ex}

\begin{ex}[A plane curve of type $(4,6;6,3)$]\label{ex:6463}
We construct a plane curve of type $(4,6;6,3)$. 
Let $E$ be a smooth cubic, let $O\in E$ be an inflectional point, and let  $P_1,P_2,P_3\in E$ be general points. 
There is a conic $C_2$ through $P_1,P_2,P_3$ and tangent to $E$ at $O$, 
and $E$ and $C_2$ intersect at $P_4$ other than $P_1,P_2,P_3,O$. 
Since $\sum_{i=1}^4P_i+2O\sim 2L|_E$ as divisors on $E$, we obtain 
\[ 3(P_1+P_2+P_3+P_4)\sim 12\,O\sim 4\,L|_E, \]
where $L$ is a general line on $\PP^2$. 
Hence there exists a quartic $C_4^\prime$ such that $C_4^\prime|_E=3\sum_{i=1}^4P_i$. 
Let $f_3,f_4^\prime$ be homogeneous polynomials defining $E,C_4^\prime$, respectively. 
We can choose a linear polynomial $g_1$ so that the curve $C_4$ defined by $f_4:=f_4^\prime+f_3g_1=0$ is smooth at $C_4^\prime\cap E$. 
By the same argument of the proof of Proposition~\ref{prop:type_2smooth}, $f_3^2+f_4g=0$ defines a smooth sextic $B$ for a general homogeneous polynomial $g$ of degree $2$, and $\WCC_3:=C_4+B$ is of type $(4,6;6)$. 
Note that $B\cap C_4=C_4\cap E=\{P_1,\dots,P_4\}$. 
Since $P_1,P_2,P_3\in B\cap C_4$ are not collinear, $\WCC_3$ is not of type $(4,6;6,1)$. 
If $\WCC_3$ is of type $(4,6;6,2)$, then there is a conic $C_2^\prime$ tangent to $C_4$ with multiplicity $2$ at each $P\in B\cap C_4$; 
this implies that $C_2^\prime.E\geq 8$, which is a contradiction. 
Therefore, $\WCC_3$ is of type $(4,6;6,3)$. 
\end{ex}

\begin{ex}[A plane curve of type $(4,6;6,6)$]\label{ex:6466}
It is more difficult to construct a plane curve of type $(4,6;6,6)$. Let us sketch how to construct such a curve.
Let $C_3$ be a nodal cubic and $O\in C_3$ an inflectional point. We consider the additive group structure on $C_3$ with $O$ being the zero element. We denote the addition on $C_3$ by $\dot+$. Let us consider an irreducible quartic $C_4$ such that 
$C_3\cap C_4=\{P_1,P_2\}$, where $P_1, P_2$ are not torsion points of $C_3$,  $(C_3\cdot C_4)_{P_i}=6$ and $\sing(C_4)=\{P_0\}$
where $(C_4,P_0)$ is of type $\mathbb{A}_3$. Such a curve exists, however we skip the details since the actual
example has coefficients in a number field of degree~$4$. Then
$Q:=P_1\dot+P_2$ is an element of order~$6$. These curves satisfy the following properties:
there is no smooth conic $C_2$ such that $(C_2\cdot C_4)_{P_1}=(C_2\cdot C_4)_{P_2}=2$
and $(C_2\cdot C_4)_{P_0}=4$ (We omit the details due to long computation.); moreover, there is no smooth conic $C'_2$ such that $(C'_2\cdot C_4)_{P_1}=(C'_2\cdot C_4)_{P_2}=3$
and $(C'_2\cdot C_4)_{P_0}=2$ (from the group law of the cubic).

Let us choose coordinates such that
$P_0=[0:1:0]$ and its tangent line is $Z=0$. Let $\sigma:\PP^2\to\PP^2$ be the rational map defined by $\sigma([x:y:z])=[x z: y^2: z^2]$. 
Let $C_6$, $D_4$ be the strict transforms 
by $\sigma$ of $C_3$, $C_4$, respectively. Here, the subindices indicate their degrees.
Choose a generic conic $D_2$, and a generic element $D_6$ of the 
pencil generated by $C_6$ and $D_4+D_2$.
The above properties imply that $\mcC_6:=D_4+D_6$ is of type $(4,6;6,6)$.
These computations can be found at
\url{https://github.com/enriqueartal/TorsionDivisorsZariskiPairs}, and can be checked using 
\texttt{Sagemath}~\cite{sage96} or \texttt{Binder}~\cite{binder}.
\end{ex}

Finally we prove Theorem~\ref{thm:example}. 

\begin{proof}[Proof of Theorem~\ref{thm:example}]
	By Examples~\ref{ex:6461_6462}, \ref{ex:6463} and \ref{ex:6466}, there are plane curves of type $(4,6;6,1)$, $(4,6;6,2)$, $(4,6;6,3)$ and $(4,6;6,6)$. Therefore there exists a Zariski $4$-tuple by Corollary~\ref{cor:example}. 
\end{proof}

\noindent
{\bf Acknowledgements:}
The final step of this paper was done during the authors visit in Universit\'e de Pau et des Pays de l'Adour. 
They thank V.~Florens and Universit\'e de Pau et des Pays de l'Adour for their hospitality. 
Also the last three authors thank E.~Artal Bartolo and Universidad de Zaragoza for the support to their stay in Zaragoza.

\providecommand{\bysame}{\leavevmode\hbox to3em{\hrulefill}\thinspace}
\providecommand{\MR}{\relax\ifhmode\unskip\space\fi MR }
\providecommand{\MRhref}[2]{%
  \href{http://www.ams.org/mathscinet-getitem?mr=#1}{#2}
}
\providecommand{\href}[2]{#2}

%
%
%

\end{document}